\documentclass[12pt]{article}
\usepackage{amssymb}
\usepackage{amsthm}
\usepackage{graphicx}
\usepackage{cite}
\usepackage{amsmath,amsthm}
\usepackage{url}
\usepackage{enumerate}
\usepackage[final]{hyperref}
\usepackage{srcltx}         
\theoremstyle{plain}
\newtheorem{theorem}{Theorem}
\newtheorem{property}{Property}
\newtheorem{corollary}{Corollary}
\theoremstyle{definition}
\newtheorem{defAlph}{Definition}
\newtheorem{definition}{Definition}
\newtheorem{example}{Example}

\begin{document}

\title{Extremal Indices in the Series Scheme\\ and their Applications\footnote{This is a 
translation (from Russian) of https://doi.org/10.14357/19922264150305 }}

\author{A. V. Lebedev\\ \small
Faculty of Mechanics and Mathematics, Lomonosov Moscow State University,\\[-3pt]
\small Leninskiye Gory, Moscow 119991, Russian Federation}

\date{}

\maketitle

\begin{abstract}
We generalize the concept of extremal index of a stationary random sequence to the
series scheme of identically distributed random variables with random series sizes
tending to infinity in probability. We introduce new extremal indices through two
definitions generalizing the basic properties of the classical extremal index. We
prove some useful properties of the new extremal indices. We show how the behavior of
aggregate activity maxima on random graphs (in information network models) and the
behavior of maxima of random particle scores in branching processes (in biological
population models) can be described in terms of the new extremal indices. We also
obtain new results on models with copulas and threshold models. We show that the new
indices can take different values for the same system, as well as values greater than
one.

\end{abstract}

\section{Introduction}

The extremal index of a stationary (in a narrow sense) random sequence $\{\xi_n\}$ is
defined as follows \cite[Section~3.7]{LLR}.

\renewcommand{\thedefAlph}{\Alph{defAlph}}

\begin{defAlph}\label{defA}\hskip-1ex\footnote{In what follows, by $\vee$ we
denote the maximum; by $\wedge$, the minimum; by an overbar above a distribution
function, we denote its tail: ${\bar F}(x)=1-F(x)$; by $f^{-1}$, the inverse function
of $f$, and for a distribution function, the generalized inverse:
$F^{-1}(y)=\inf\{x:F(x)\ge y\}$; by $f(x)^n$ we denote the $n$th power of
$f(x)$.}\hskip1ex Let $\xi_n$, $n\ge 1$, have distribution $F$, and let
$M_n=\vee_{k=1}^n\xi_k$. If for any $\tau>0$ there exists a number sequence
$u_n(\tau)$ such that $n{\bar F}(u_n(\tau))\to \tau$ and ${\bf P}(M_n\le
u_n(\tau))\to e^{-\theta\tau}$, then $\theta$ is said to be the \emph{extremal
index}.
\end{defAlph}

Any value of $\theta\in [0,1]$ is possible here.

Note that if we take maxima ${\hat M}_n$ of a sequence of independent random
variables with the same distribution $F$, then
$$
\lim_{n\to\infty}{\bf P}({\hat M}_n\le u_n(\tau))=e^{-\tau},
$$
which implies
\begin{equation}\label{Mfunc}
\lim_{n\to\infty}{\bf P}(M_n\le u_n(\tau))=\left(\lim_{n\to\infty}{\bf P}({\hat
M}_n\le u_n(\tau))\right)^\theta;
\end{equation}
i.e., limiting distribution functions for $M_n$ and ${\hat M}_n$ have a power-law
dependence,
\begin{equation}\label{Mtheta}
\lim_{n\to\infty}{\bf P}(M_n\le u_n(\tau))=\lim_{n\to\infty}{\bf P}({\hat M}_{[\theta
n]}\le u_n(\tau)),\quad\theta>0;
\end{equation}
i.e., $M_n$ asymptotically grows as the maximum of $[\theta n]$ independent random
variables as $n\to\infty$ and
\begin{equation}\label{Mhat}
\lim_{n\to\infty}{\bf P}(M_n\le u_n(\tau))\ge \lim_{n\to\infty}{\bf P}({\hat M}_n\le
u_n(\tau));
\end{equation}
i.e., $M_n$ is stochastically not greater than the maximum of independent random
variables (in the limit).

The interest to the extremal index is partly due to the fact that its existence
preserves the extremal type of the limiting distribution of the maxima. Recall that
if for some number sequences $a_n>0$ and $b_n$, $n\ge 1$, and for a nondegenerate
distribution $G$ there exists a limit
$$
\lim_{n\to\infty}{\bf P}({\hat M}_n\le a_nx+b_n)=G(x),\quad \forall x\in\mathbb{R},
$$
then $G$ belongs to one of three extremal types, namely: $G(x)=G_i(ax+b)$ for some
$a>0$ and $b$, where $G_1(x)=\exp\{-e^{-x}\}$ (Gumbel type),
$G_2(x)=\exp\{-x^{-\alpha}\}$, $x>0$, $\alpha>0$ (Fr\'echet type), and
$G_3(x)=\exp\{-(-x)^\alpha\}$, $x\le 0$, $\alpha>0$ (Weibull type). Such
distributions $G$ are referred to as max-stable or extreme value distributions. For
any $s>0$ there exist $a(s)>0$ and $b(s)$ such that $G^s(x)=G(a(s)x+b(s))$. Thus,
raising the limiting distribution function to the power $\theta>0$, which arises
because of property \eqref{Mfunc}, preserves the extremal type.

One of the interpretations of the extremal index consists in the fact that passages
over a high level in a sequence occur not one at a time but in batches (clusters) of
average size $1/\theta$. In applications, this can mean natural disasters, failures
in technical systems, data losses in information transmission, financial losses, etc.
Clearly, if such events happen several times in succession, this is much more
dangerous than single occurrences and must be taken into account in risk management.

For more details on this subject, see \cite{LLR, Gal, EKM, HF}.

Since 1980s, active investigations in this field have been made in two main
directions: finding the extremal index for various random sequences and constructing
statistical estimators for the extremal index based on observations.

For a survey of results and references, one can see, e.g., \cite[Section~8.1]{EKM}
and \cite[Section~5.5]{HF}. Section~1.2 in the dissertation \cite{Novak-2013} was
specially devoted to generalizations of the classical notion of the extremal index
and its statistical estimation. In particular, the following definition can be given.

\begin{defAlph}\label{defB}
Let $\xi_n$, $n\ge 1$, have distribution $F$, and let $M_n=\vee_{k=1}^n\xi_k$. If for
each number sequence $u_n$, $n\ge 1$, such that
$$
0<\liminf_{n\to\infty}n{\bar F}(u_n)\le\limsup_{n\to\infty}n{\bar F}(u_n)<\infty,
$$
we have ${\bf P}(M_n\le u_n)-F(u_n)^{\theta n}\to 0$, $n\to\infty$, then $\theta$ is
called the \emph{extremal index}.
\end{defAlph}

This definition allows to extend the notion of the extremal index to some stationary
sequences of random variables with discrete distributions (for instance, geometric),
and for continuous distributions it is equivalent to Definition~\ref{defA}.

Papers \cite{Markovich1, Markovich2} were devoted to the analysis of extrema and
passage over a high level related to telecommunication models, and in
\cite{Markovich-new} distributions and dependence of extrema in network sampling
processes were studied, including the extremal indices.

In the dissertation \cite{Gold}, new interesting results are obtained concerning
extremal indices of sequences of the form
$$
X_n=A_nX_{n-1}+B_n,
$$
where $(A_n,B_n)$, $n\ge 1$, are independent random pairs taking values in
$\mathbb{R}_+^2$. In some cases, extremal indices and distributions of cluster sizes
are explicitly computed, and in the more general case, upper and lower bounds on the
extremal index are obtained. Continuity of the extremal index is proved with respect
to a certain convergence of distributions of the coefficients. Indices of
multivariate sequences with heavy tails are introduced and analyzed. A part of the
obtained results is presented in \cite{Gold1, Gold2}.

However, in practice it is necessary to study maxima on more complex structures than
the set of natural numbers. Difficulties related with this were discussed as far back
as in \cite[Sections~3.9 and~3.12]{Gal}. For example, if we consider lifetimes of
components of a compound system (reliability scheme), it is not clear how we can
enumerate them so that the model of a~stationary sequence can be used, nor whether
this is possible in principle. A little simpler is the case of random fields.

The extremal index can naturally be extended from random sequences to random fields
on lattices $\mathbb{N}^d$ \cite{Choi}. Consider, for example, a random field
$\{\xi_{n_1,n_2}\}$ in $\mathbb{N}^2$, and let
$M_{n_1,n_2}=\vee_{k_1=1}^{n_1}\vee_{k_2=1}^{n_1}\xi_{k_1,k_2}$. If for each
$\tau>0$ there exists $u_{n_1,n_2}(\tau)$ such that $n_1n_2{\bar
F}(u_{n_1,n_2}(\tau))\to \tau$ and ${\bf P}(M_{n_1,n_2}\le u_{n_1,n_2}(\tau))\to
e^{-\theta\tau}$, then $\theta$ is called the extremal index. To the computation of
the extremal index of a random field in $\mathbb{N}^2$, the paper \cite{Exp1} is
devoted; in \cite{Exp2}, the asymptotic location of the maximum of a random field
with a certain extremal index was studied.

Since the above-mentioned results are not directly related to the subject of the
present paper, we shall not dwell on them further, but we only emphasize the
relevance of studying the extremal index in various models and applications.

Below we present a new generalization of the extremal index to a series scheme with
random sizes, which allows one to deal with a wider class of stochastic structures.
Moreover, we give two different definitions.

Consider a collection of random variables $\xi_{n,m}$, $n\ge 1$, $m\ge 1$, with
distributions $F_n$ (here~$n$ is the series number, and $m$ is the number of the
random variable in a series) and also a sequence of integer random variables (series
sizes) $\nu_n\stackrel{P}{\to}+\infty$, $n\to\infty$, and let
$M_n=\vee_{m=1}^{\nu_n}\xi_{n,m}$.

\begin{definition}\label{def1}\hskip-1ex\footnote{As compared to Definition~\ref{defA},
we have made the change of variables $s=e^{-\tau}$ and accordingly redefined the
functions $u_n$, $n\ge 1$.}\hskip1ex Let for each $s\in (0,1)$ there exist a sequence
$u_n(s)$ such that ${\bf E}F_n(u_n(s))^{\nu_n}\to s$ and ${\bf P}(M_n\le
u_n(s))\to\psi(s)$, $n\to\infty$. The we call $\psi$ the \emph{extremal function}. If
$\psi(s)=s^\theta$, we call $\theta$ the \emph{extremal index}.
\end{definition}

In the general case, we can define partial indices
$$
\theta^+=\sup_{s\in (0,1)}\log_s\psi(s),\qquad \theta^-=\inf_{s\in
(0,1)}\log_s\psi(s);
$$
then $\theta^+\ge \theta^-$ and $s^{\theta^+}\le\psi(s)\le s^{\theta^-}$, $s\in
(0,1)$.

The essence of Definition~\ref{def1} consists in comparing the limiting distributions
of $M_n$ and of the maxima ${\hat M}_n$ of $\nu_n$ independent random variables (the
number $\nu_n$ being independent of them) under the same normalization given by the
condition ${\bf P}({\hat M}_n\le u_n(s))\to s$, $n\to\infty$. Thus, we generalize
property \eqref{Mfunc}.

It is clear that the indices, as above, take nonnegative values, but the upper
boundary 1 is removed, at least for $\theta^+$, as will be shown below
(Examples~\ref{ex5.5}, \ref{ex6.1}, and~\ref{ex6.2}). This happens because inequality
\eqref{Mhat} can be violated. The maxima over series can grow asymptotically faster
than the maxima of independent random variables taken in the same quantity, which
corresponds to the case $\psi(s)<s$, $s\in (0,1)$.

\begin{definition}\label{def2}
Let for each $s\in (0,1)$ there exist a sequence $u_n(s)$ such that ${\bf
E}F_n(u_n(s))^{\nu_n}\to s$ and ${\bf P}(M_n\le u_n(s))-{\bf
E}F_n(u_n(s))^{\theta\nu_n}\to 0$, $n\to\infty$; then $\theta$ is called the
\emph{extremal index}.
\end{definition}

The essence of Definition~\ref{def2} consists in choosing a value of $\theta$ such
that the limiting distributions of $M_n$ and of the maxima of $[\theta\nu_n]$
independent random variables (the number $[\theta\nu_n]$ being independent of them)
coincide under the same normalization as in Definition~\ref{def1} (for $\theta>0$).
Thus, property \eqref{Mtheta} is generalized.

The existence of an extremal index in the sense of Definition~\ref{def2} actually
means that the extremal function from Definition~\ref{def1} admits the representation
$$
\psi(s)=\lim_{n\to\infty}{\bf E}F_n(u_n(s))^{\theta\nu_n}.
$$

A question arises of why we give two definitions; cannot we do with only one? Indeed,
in many cases both definitions of the index are equivalent: they exists and are equal
to each other (Section~\ref{sec3}, Example~\ref{ex5.1}). However, it also happens
that both indices do not exist but in the sense of Definition~\ref{def1} there exists
an extremal function and partial indices (Examples~\ref{ex5.2}, \ref{ex5.3},
\ref{ex6.1}, and~\ref{ex6.2}); it also happens that the index in the sense of
Definition~\ref{def2} exists while the index in the sense of Definition~\ref{def1}
does not exist and both the extremal function and partial indices again exist
(Section~\ref{sec4}); finally, there can be a surprising situation where both indices
do exist but take \emph{different\/} values (Example~\ref{ex5.4}). Thus, these are
indeed two different characteristics of the system which do not reduce to a single
one.

Note that maxima in a series scheme were previously considered in \cite{Sav} for
random variables related by IT-copulas (individuated $t$-copulas), and conditions
were derived under which the maxima asymptotically grow as in the case of independent
variables, i.e., in our terms, $\theta=1$.

To avoid ambiguity in terminology, below we speak about extremal indices of a
\emph{system\/} (of random variables), denoted by $\{\xi_{n,m};\nu_n\}$.

In Section~\ref{sec2} we prove basic properties of the extremal indices; we present
their applications to information network models in Section~\ref{sec3}, to models of
biological populations in Section~\ref{sec4}, to models with copulas in
Section~\ref{sec5}, and to threshold models in Section~\ref{sec6}.

\section{Basic Properties of Extremal Indices}\label{sec2}

The extremal indices have the following properties.

\begin{property}\label{pr1}
Let\/ $\eta_n$, $n\ge 1$, be a stationary sequence with extremal index\/ $\theta$ in
the sense of Definition~\ref{defA}. Put\/ $\xi_{n,m}=\eta_m$, $m\ge 1$, and consider
an integer sequence $l_n\to +\infty$\textup; then the system\/ $\{\xi_{n,m};l_n\}$
has extremal index\/ $\theta$ in the sense of Definitions~\ref{def1} and~\ref{def2}.
\end{property}

\begin{proof}
Denote by $u^0_n(\tau)$, $n\ge 1$, the sequence that exists by Definition~\ref{defA},
and let $u_n(s)=u^0_{l_n}(-\ln s)$; then $F(u_n(s))^{l_n}\to s$, ${\bf P}(M_n\le
u_n(s))\to s^\theta$, and $F(u_n(s))^{\theta l_n}\to s^\theta$, which gives the same
extremal index according to both definitions.
\end{proof}

\begin{property}\label{pr2}
Let a system\/ $\{\xi_{n,m};\nu_n\}$ have an extremal index in the sense of one of
Definitions~\ref{def1} or~\ref{def2}\/ \textup(or an extremal function\textup), and
consider a sequence of functions $g_n(x)$, $n\ge 1$, that are continuous and strictly
increasing on the set of points of increase of\/ $F_n$. Put\/
${\tilde\xi}_{n,m}=g_n(\xi_{n,m})$\textup; then the system\/
$\{{\tilde\xi}_{n,m};\nu_n\}$ has the same extremal index\/ \textup(extremal
function\textup).
\end{property}

\begin{proof}
For the new system, ${\tilde F}_n(x)=F_n(g^{-1}_n(x))$. Let ${\tilde
u}_n(s)=g_n(u_n(s))$; then ${\tilde F}_n({\tilde u}_n(s))=F_n(u_n(s))$ and ${\bf
P}({\tilde M}_n\le {\tilde u}_n(s))= {\bf P}(M_n\le u_n(s))$, so all the limits (in
Definitions~\ref{def1} and~\ref{def2}) do not change.
\end{proof}

\begin{property}\label{pr3}
Let a system\/ $\{\xi_{n,m};\nu_n\}$ have an extremal index in the sense of one of
Definitions~\ref{def1} or~\ref{def2}, and let there exist a sequence $c_n\to +\infty$
such that\/ $\nu_n/c_n\stackrel{P}{\to} 1$, $n\to\infty$\textup; then the system has
the same extremal index in the sense of the other definition.
\end{property}

\begin{proof}
In this case ${\bf E}F_n(u_n(s))^{\nu_n}={\bf E}(F_n(u_n(s))^{c_n})^{\nu_n/c_n}\to
s\in (0,1)$ implies $F_n(u_n(s))^{c_n}\to s$ and ${\bf E}F_n(u_n(s))^{r\nu_n}\to
s^r$, $r\ge 0$, $n\to\infty$. Thus, if $\theta$ is the extremal index in the sense of
Definition~\ref{def1}, then ${\bf P}(M_n\le u_n(s))\to s^\theta$ implies ${\bf
E}F_n(u_n(s))^{\theta\nu_n}\to s^\theta$, and therefore $\theta$ is the extremal
index in the sense of Definition~\ref{def2}. Vice versa, if $\theta$ is the extremal
index in the sense of Definition~\ref{def2}, then ${\bf
E}F_n(u_n(s))^{\theta\nu_n}\to s^\theta$ implies ${\bf P}(M_n\le u_n(s))\to
s^\theta$, and therefore $\theta$ is the extremal index in the sense of
Definition~\ref{def1}.
\end{proof}

\begin{property}\label{pr4}
Consider a system\/ $\{\xi_{n,m};\nu_n\}$ with extremal index\/ $\theta>0$ in the
sense of Definition~\ref{def2} such that\textup:
\begin{enumerate}[\rm(a)]\addtolength{\itemsep}{-5pt}
\item\vskip-5pt
$F_n\equiv F$\textup;
\item
For some max-stable law $G$ and functions $a(r)>0$ and\/ $b(r)$, $r>0$, we
have
$$
F^r(a(r)x+b(r))\to G(x),\quad r\to\infty;
$$
\item
There exists a sequence $c_n\to +\infty$ such that
$\nu_n/c_n\stackrel{d}{\to}\zeta>0$, $n\to\infty$\textup;
\item
In Definition~\ref{def2} we may take $u_n(s)=A_nH^{-1}(s)+B_n$, where $A_n=a(c_n)$,
$B_n=b(c_n)$, and\/ $H(x)$ is a continuous distribution function.
\end{enumerate}\vskip-5pt
Then\footnote{If $\zeta$ is degenerate and is equal to a constant $c>0$ a.s., then,
clearly, ${\bf E}u^\zeta=u^c$.} $H(x)={\bf E}G(x)^\zeta$ and
$$
{\bf P}(M_n\le A_nx+B_n)\to H(ax+b),\quad n\to\infty,
$$
where $a>0$ and\/ $b$ are determined by the identity $G(x)^\theta=G(ax+b)$.
Furthermore, the extremal function in the sense of Definition~\ref{def1} is\/
$\psi(s)=H(aH^{-1}(s)+b)$.
\end{property}

\begin{proof}
By \cite[Corollary~1.3.2]{LLR}, for any max-stable law there exist $a>0$ and $b$ such
that $G(x)^\theta=G(ax+b)$. Let $x=H^{-1}(s)$. Since
$$
{\bf E}F(A_nx+B_n)^{\nu_n}={\bf E}(F(a_{l_n}x+b_{l_n})^{l_n})^{\nu_n/l_n}\to {\bf
E}G(x)^\zeta,\quad n\to\infty,
$$
we have $H(x)={\bf E}G(x)^\zeta$. Then
$$
{\bf E}F(A_nx+B_n)^{\theta\nu_n}\to {\bf E}G(x)^{\theta\zeta}={\bf E}G(ax+b)^\zeta=
H(ax+b),\quad n\to\infty.
$$
By Definition~\ref{def2} this implies ${\bf P}(M_n\le A_nx+B_n)\to H(ax+b)$,
$n\to\infty$, and by Definition~\ref{def1} we obtain $\psi(s)=H(aH^{-1}(s)+b)$.
\end{proof}

\begin{property}\label{pr5}
Consider a system\/ $\{\xi_{n,m};\nu_n\}$ such that\textup:
\begin{enumerate}[\rm(a)]\addtolength{\itemsep}{-5pt}
\item\vskip-5pt
$F_n\equiv F$\textup;
\item
There exists a sequence $c_n\to +\infty$ such that\/
$\nu_n/c_n\stackrel{d}{\to}\zeta>0$, $n\to\infty$\textup;
\item
For a continuous distribution $G$ and coefficients $A_n>0$ and $B_n$, we have
$$
\begin{gathered}
F(A_nx+B_n)^{c_n}\to G(x),\\ {\bf P}(M_n\le A_nx+B_n)\to {\bf
E}G(x)^{\theta\zeta},\quad n\to\infty.
\end{gathered}
$$
\end{enumerate}\vskip-5pt
Then\/ $\theta$ is the extremal index in the sense of Definition~\ref{def2}.
\end{property}

\begin{proof}
First of all, we have ${\bf E}F(A_nx+B_n)^{\nu_n}\to {\bf E}G(x)^\zeta$. Denote
$H(x)={\bf E}G(x)^\zeta$; this is a~continuous function running over all values in
$(0,1)$. Put $x=H^{-1}(s)$, $u_n(s)=A_nx+B_n$; then ${\bf E}F(u_n(s))^{\nu_n}\to s$
and ${\bf E}F(u_n(s))^{\theta\nu_n}\to {\bf E}G(x)^{\theta\zeta}$. Since we also have
${\bf P}(M_n\le u_n(s))\to {\bf E}G(x)^{\theta\zeta}$ by the condition, ${\bf
P}(M_n\le u_n(s))-{\bf E}F(u_n(s))^{\theta\nu_n}\to 0$, $n\to\infty$, so $\theta$ is
the extremal index in the sense of Definition~\ref{def2}.
\end{proof}

\begin{property}\label{pr6}
Consider a system\/ $\{\xi_{n,m}; l_n\}$ with extremal index\/ $\theta$ in the sense
of one of the definitions such that\textup:
\begin{enumerate}[\rm(a)]\addtolength{\itemsep}{-5pt}
\item\vskip-5pt
$F_n\equiv F$ is a continuous distribution\textup;
\item
$l_n$, $n\ge 1$, is an integer sequence, $l_n\to +\infty$, $l_n\sim n^{\alpha}L(n)$,
$n\to\infty$, $\alpha>0$, $L(x)$ being a slowly varying function on\/ $\mathbb{R}_+$.
\end{enumerate}\vskip-5pt
Let\/ $\nu_n/l_n\stackrel{P}{\to} 1$, $n\to\infty$\textup; then the system\/
$\{\xi_{n,m}; \nu_n\}$ has the same extremal index in the sense of both definitions.
\end{property}

\begin{proof}
By Property~\ref{pr3} each of the systems has equal extremal indices in the sense of
both definitions. Denote by $M_n$ the maxima for $\{\xi_{n,m}; l_n\}$, and by
${\tilde M}_n$, for $\{\xi_{n,m}; \nu_n\}$.

For any $\rho>0$ we have $l_{[n\rho]}\sim\rho^\alpha l_n$, $n\to\infty$. Therefore,
$\nu_n/l_n\stackrel{P}{\to} 1$, $n\to\infty$, implies
$$
{\bf P}(M_{[n(1-\varepsilon)]}\le {\tilde M}_n\le M_{[n(1+\varepsilon)]})\to 1,\quad
n\to\infty,
$$
for any $\varepsilon>0$. Since $F(u_n(s))^{l_n}\to s$, we have
$$
u_n(s)=F^{-1}(1+(1+o(1))(\ln s)/l_n),\quad n\to\infty,
$$
and for any $\rho>0$ by virtue of Definition~\ref{def1} we have
$$
\begin{aligned}
{\bf P}(M_{[n\rho]}\le u_n(s))&={\bf P}(M_{[n\rho]}\le F^{-1}(1+(1+o(1))(\ln
s)/l_n))\\ &={\bf P}(M_{[n\rho]}\le F^{-1}(1+(1+o(1))(\ln
s^{1/\rho^\alpha})/l_{[n\rho]}))\to s^{\theta/\rho^\alpha},\quad n\to\infty.
\end{aligned}
$$
Letting $\rho=1\pm\varepsilon$, $\varepsilon>0$, we obtain
$$
s^{\theta/(1+\varepsilon)^\alpha}\le\liminf_{n\to\infty}{\bf P}({\tilde M}_n\le
u_n(s))\le \limsup_{n\to\infty}{\bf P}({\tilde M}_n\le u_n(s))\le
s^{\theta/(1-\varepsilon)^\alpha}.
$$
Passing to the limit as $\varepsilon\to 0$, we obtain $\lim_{n\to\infty}{\bf
P}({\tilde M}_n\le u_n(s))=s^\theta$; hence, $\theta$ is the extremal index of
$\{\xi_{n,m}; \nu_n\}$ in the sense of Definition~\ref{def1}.
\end{proof}

\begin{property}\label{pr7}
Consider a system\/ $\{\xi_{n,m};\nu_n\}$ with extremal index\/ $\theta$ in the sense
of Definition~\ref{def1} for which $F_n\equiv F$ is a continuous distribution,
$\nu_n/n\stackrel{P}{\to} c>0$, $n\to\infty$, and consider a random integer
sequence\/ $\eta_n$ independent of\/ $\{\xi_{n,m};\nu_n\}$ and a sequence\/ $\mu_n\to
+\infty$ such that\/ $\eta_n/\mu_n\stackrel{d}{\to}\zeta>0$, $n\to\infty$. Put\/
${\tilde\xi}_{n,m}=\xi_{\eta_n,m}$, ${\tilde\nu_n}=\nu_{\eta_n}$\textup; then the
system\/ $\{{\tilde\xi}_{n,m};{\tilde\nu}_n\}$ has extremal index\/ $\theta$ in the
sense of Definition~\ref{def2}.
\end{property}

\begin{proof}
For the sequence $u_n(s)$, from Definition~\ref{def1} for the system
$\{\xi_{n,m};\nu_n\}$ we have the convergence ${\bf E}F(u_n(s))^{\nu_n}\to s$, whence
$F(u_n(s))^{cn}\to s$, $n\to\infty$, so that
$$
u_n(s)=F^{-1}(1+(1+o(1))(\ln s)/(cn)),\quad n\to\infty.
$$
For any $x\in (0,1)$ we have
$$
\begin{aligned}
{\bf E}F(u_{[\mu_n]}(x))^{{\tilde\nu}_n}&={\bf E}F(u_{[\mu_n]}(x))^{{\tilde\nu}_n}\\
&={\bf E}F(u_{[\mu_n]}(x))^{(\nu_{\eta_n}/\eta_n)(\eta_n/\mu_n)\mu_n}\to {\bf
E}x^\zeta,\quad n\to\infty.
\end{aligned}
$$
Denote $H(x)={\bf E}x^\zeta$, $x=H^{-1}(s)$, and ${\tilde u}_n(s)=u_{[\mu_n]}(x)$;
then
$$
{\bf E}F({\tilde u}_n(s))^{{\tilde\nu}_n}\to s,\qquad {\bf E}F({\tilde
u}_n(s))^{\theta{\tilde\nu}_n}\to {\bf E}x^{\theta\zeta},\quad n\to\infty.
$$
On the other hand, by Definition~\ref{def1} for the system $\{\xi_{n,m};\nu_n\}$ we
obtain
$$
\begin{aligned}
{\bf P}({\tilde M}_n\le{\tilde u}_n)&={\bf P}(M_{\eta_n}\le u_{[\mu_n]}(x))\\ &={\bf
P}(M_{\eta_n}\le F^{-1}(1+(1+o(1))(\ln x)/(c\mu_n)))\\ &={\bf P}(M_{\eta_n}\le
F^{-1}(1+(1+o(1))(\ln x^{\eta_n/\mu_n})/(c\eta_n)))\to {\bf E}x^{\theta\zeta},\quad
n\to\infty.
\end{aligned}
$$
Hence, ${\bf P}({\tilde M}_n\le{\tilde u}_n)-{\bf E}F({\tilde
u}_n(s))^{\theta{\tilde\nu}_n}\to 0$, $n\to\infty$, and $\theta$ is the extremal
index of $\{{\tilde\xi}_{n,m};{\tilde\nu}_n\}$ in the sense of Definition~\ref{def2}.
\end{proof}

Let us give some comments on the above-proved properties.

Property~\ref{pr1} means that the introduced indices are indeed generalizations of
the classical extremal index (in the sense of Definition~\ref{defA}) and coincide
with it if, as series, we take deterministically growing segments of a stationary
sequence.

Property~\ref{pr2} means invariance of the extremal indices under continuous strictly
increasing transforms of a series. This means, for instance, that in the case of
continuous random variables they all can be reduced to the uniform distribution on
$[0,1]$ by a transform with $g_n=F_n$. A~similar property holds for the classical
extremal index (in the sense of Definition~\ref{defA}) when we speak about a single
continuous strictly increasing transform applied to all elements of the sequence.

Property~\ref{pr3} specifies a restriction on the randomness of series sizes under
which both new indices are equivalent. The sizes must asymptotically grow
equivalently to a nonrandom sequence.

Property~\ref{pr4} generalizes the well-known statement for the classical extremal
index \cite[Corollary~3.7.3]{LLR}: the limiting distribution of maxima of a
stationary sequence has the same extremal type as the limiting distribution of maxima
of independent random variables with the same marginal distribution. In this case the
limiting law need not be max-stable, but its type is preserved. It is max-stable only
for a degenerate random variable $\zeta$ (a constant), i.e., under the conditions of
Property~\ref{pr3}.

Property~\ref{pr5} allows one to interpret a parameter of the limiting distribution
of maxima of dependent random variables as the extremal index in the sense of
Definition~\ref{def2}.

Property~\ref{pr6} provides a sufficient condition on the growth rate of series sizes
under which the indices for random and nonrandom variables coincide.

Property~\ref{pr7} shows that randomization over a randomly growing series number
allows to pass from the extremal index in the sense of Definition~\ref{def1} to the
same index in the sense of Definition~\ref{def2}.

\section{Applications to Information Network Models}\label{sec3}

In the papers \cite{Leb3, Leb4, Leb-Nc} the author considered maxima of aggregate
activity in information networks described by power-law random graphs (also
referred to as Internet or Internet-type graphs in the Russian literature).

As examples of recent works of Russian authors on power-law graphs, we refer the
reader to \cite{Pavl, Leri} and a survey \cite{Raig}. We also recommend a foreign
electronic textbook \cite{Hofstad}.

Consider the following example of an information network model \cite{Leb4}.

Let each network node have an individual random information activity (rate of
information production). We assume that activities of the nodes are independent and
identically distributed and that their distribution~$A$ has a heavy (regularly
varying) tail, i.e.,  ${\bar A}(x)\sim x^{-a}L(x)$, $x\to\infty$, $a>0$, where $L(x)$
is a slowly varying function. Activities and degrees of vertices (nodes) are assumed
to be independent, and this assumption is essential.

Consider the model of a directed random graph where edge directions correspond to
directions of information transmission. Assume that we have $n$ vertices and there
are independent nonnegative integer random variables $K_1,\ldots, K_n$ with the same
distribution defined by the probabilities $p_k\sim ck^{-\beta}$, $k\to\infty$,
$\beta>2$. Let $D_i=\min\{K_i,n-1\}$. For the $i$th vertex, choose at random
(equiprobably and independently of the choices for other vertices) $D_i$ different
vertices among the others (except for the $i$th vertex) and draw edges from them to
the $i$th vertex. The resulting graph can be regarded as a power-law graph in the
sense that the number of incoming edges has asymptotically a power-law distribution.
The aggregate activity at a node is in this case defined to be the sum of its own
activity and activities of all nodes from which information is coming (its incoming
neighbors).

We emphasize that the activity has no relation to the notions of ``quality'' or
``weight'' of a vertex, used in modern random graphs models. In our case a graph is
formed by the above-described algorithm, and individual activities are independent
complements to the graph. As regards social networks, it may happen that a user
writes much but reads little (or is read by few), or, on the contrary, writes little
but reads much (or is read by many). As for the aggregate activity, it may happen to
be large because of few incoming neighbors (or even a~single neighbor) with large
individual activities, or, on the contrary, small but with a large number of
incoming neighbors having low individual activities. As is well known, for heavy
tails it is typical that large values of a sum are attained at the expense of one
large (maximal) summand. This property is extended in this case to sums of randomly
many summands as well.

Denote by $M_n$ the maximum of aggregate activities. Let $v(r)$ be a positive
nondecreasing function such that $r{\bar A}(v(r))\to 1$, $r\to\infty$. Note that
$v(r)$ definitely exists and is regularly varying with exponent $1/a$
\cite[Section~1.5]{Sen}.

Then, for $a<\beta-2$ if $2<\beta<3$ and for $a<(\beta-1)/2$ if $\beta\ge 3$, the
Fr\'echet limit law holds: ${\bf P}(M_n/v(n)\le x)=\exp\{-x^{-a}\}$, $x>0$,
$n\to\infty$. Note that this limit law is due to the fact that the maximum of
aggregate activities grows asymptotically equivalently (in probability) to the
maximum of individual activities over the network, which is proved in
\cite[Theorem~1]{Leb4}.

On the other hand, if the number of incoming neighbors is described by a random
variable~$K$ independent of the activity, the limiting distribution $F$ of the
aggregate activity at each node has a tail
$$
{\bar F}(x)\sim (1+{\bf E}K){\bar A}(x),\quad x\to\infty,
$$
under the condition (which is fulfilled in this case)
$$
{\bf E}K^{1\vee (a+\varepsilon)}<\infty,\quad\varepsilon>0,
$$
according to the results of \cite{Stam} on the distribution of a sum of randomly many
independent random variables with a heavy tail. Therefore,
$$
F(xv(n))^n\to\exp\{-(1+{\bf E}K)x^{-a}\},\quad x>0,\quad n\to\infty.
$$
Denoting $s=\exp\{-(1+{\bf E}K)x^{-a}\}\in (0,1)$, $u_n(s)=xv(n)$, we conclude that
the system of individual activities has extremal index $\theta=1/(1+{\bf E}K)$ in the
sense of Definition~\ref{def1} (and thus, also in the sense of Definition~\ref{def2}
by Property~\ref{pr3}, since $\nu_n=n$).

The value $\theta\in (0,1)$ for a sequence means that passages over a high level
occur not one at a time but in batches (clusters) of average size $1/\theta$
\cite[Section~8.1]{EKM}. In our case we may also conjecture on forming of such
clusters.

In respect to information networks, this may concern batches of nodes with high
aggregate activities caused by high individual activity of a single node that is
their common incoming neighbor.

Checking the presence of this effect in real-world networks, of course, requires
experimental investigation, which is beyond the scope of this theoretical paper.

Also, the author must admit that the choice of a random graph model in \cite{Leb4}
was determined not by its advantages in the description of real-world networks
against other modern models (e.g., scale-free models) but by a relatively simple
proof of the asymptotic equivalence of the growth of maxima of aggregate and
individual activities with the use of methods of the author's paper \cite{Leb-2005c}.
Note that merely knowing the power law for the number of incoming vertices is
absolutely insufficient, and each random graph model should be analyzed individually.
For example, the growth rate of the maximum vertex degree in the graph is of
importance. If we enforcingly cut the vertex degrees at a growing (with the number of
vertices) threshold, we can obtain a class of models with the same limiting
distributions of vertex degrees but with different growth rates of the maximum
degree, for which different constraints on $a$ will occur depending on $\beta$. Other
nuances also play a role.

Recently, the author has obtained new results for simple models with weights
\cite{Leb-Nc}. Similar models were studied in \cite[ch.~6]{Hofstad} as generalized
random graphs. It is assumed that vertices are assigned with independent weights
$w_i$, $1\le i\le n$, identically distributed as a nonnegative random variable $W$,
${\bf E}W^\beta<\infty$, $\beta\ge 1$.

In Model 1 we assume $p_i=\varphi(w_in^{-s/2})$, where $0<s\le 1$, and for $\varphi$
on $\mathbb{R}_+$ we have $0\le \varphi(x)\le \min\{1,x\}$, $\varphi(x)\sim x$, $x\to
0$. For known values of $w_i$, $1\le i\le n$, every pair of vertices $i$ and $j$ is
joined by an edge with probability $p_ip_j$ independently of other pairs. The graph
is assumed to be undirected; information is transmitted along an edge in both
directions. In respect to social networks, weights may reflect sociability of users.

In Model 2, under the same assumptions on the weights, we assume
$p_i=\varphi(w_in^{-s})$, $0<s\le 1$. For known values of $w_i$, $1\le i\le n$, the
$i$th vertex is entered by an edge from any other vertex with probability $p_i$
independently of other edges. The graph is assumed to be directed; information is
transmitted in the direction of an edge. In respect to social networks, weights may
reflect inquisitiveness of users.

In both cases the author has proved asymptotic equivalence of the growth of maxima of
aggregate and individual activities under certain restrictions on $a$ depending on
$\beta$ and $s$. For $s=1$ there exist limiting distributions of the number of
neighbors (incoming neighbors), and the obtained results can be interpreted as
existence of extremal indices: $\theta=1/(1+({\bf E}W)^2)$ in Model~1 and
$\theta=1/(1+{\bf E}W)$ in Model~2 (in the sense of both definitions), similarly to
the preceding example.

Application of the developed method to other---more complicated and popular---models
of information networks is a subject of further study.

\section{Application to Models of Biological Populations}\label{sec4}

As models of biological populations, branching processes are often used. Elements of
a population are traditionally referred to as particles. Particles may possess some
(quantitative) random scores.

In the case of living organisms, these may be size, weight, or other characteristics
such as yield of milk in cows, egg production in hens, crop capacity in plants,
sensitivity of organisms to harmful and dangerous factors, etc.

Propagation of computer viruses can also be described by branching processes.
Polymorphic computer viruses can not only propagate but also change their codes
(similarly to mutations in living organisms). As scores, one may consider certain
characteristics of the virus code or its vital activity.

In \cite{AV, Pakes}, maxima of independent random scores of particles in branching
processes were studied. As applications, \cite{AV} considered man height, and in
\cite{Pakes} horse racing was mentioned, with prize points as the score.

Let us give one more example. If there is a colony of harmful organisms with
different individual sensitivity thresholds to some factor (poison, antibiotic,
etc.), then one needs the maximum concentration to exterminate the whole colony,
since otherwise it will survive and propagate again.

In a series of works, the author considered maxima of random scores of particles in
supercritical branching processes without extinction (with finite mean and variance
of the number of descendants). Thus, \cite{Leb5} considered continuous-time
processes, and in \cite{Leb60, Leb6}, discrete-time processes were addressed.
However, scores of different particles were assumed to be independent. In \cite{Leb7}
there was for the first time studied a model with dependence of particle scores in a
generation caused by their common heredity.

First, there was considered the case where the scores have a standard normal
distribution and the correlation coefficient for scores of a pair is majorized by
$r^k$, $r\in (0,1)$, if these particles have the nearest common ancestor $k$
generations back. It was shown that maxima over generations grow asymptotically in
the same way as in the case of independent scores, which corresponds to $\theta=1$.

Next, there was considered the case where the scores have a distribution with a
regularly varying tail and the heredity is explicitly described by a linear
autoregression process of the first order:
\begin{equation}\label{lireg}
\xi_{n,m}=a\xi_{n-1,\kappa(n,m)}+b\xi^*_{n,m},\quad a\in (0,1),\quad b>0,
\end{equation}
where $\xi_{n,m}$ is the score of the $m$th particle in the $n$th generation,
$\kappa(n,m)$ is the number of the ancestor of this particle in the preceding
generation, and the random variables $\xi^*_{n,m}$, $m\ge 1$, $n\ge 1$, are
independent and have the same distribution $A$ satisfying the conditions
\begin{equation}\label{prav}
{\bar A}(x)\sim x^{-\gamma}L(x),\quad A(-x)/{\bar A}(x)\to p\ge 0, \quad
x\to\infty,\quad\gamma>0,
\end{equation}
where $L(x)$ is a slowly varying function.

In the model \eqref{lireg}, a unique stationary distribution $F$ exists. It is
assumed that all particle scores have this stationary distribution. To reveal the
role of heredity ``in pure form,'' it is desirable to ensure independence of the
score distribution from the autoregression coefficients (as was the case in the
Gaussian framework). Here we can reach this goal only for strictly stable
distributions with $0<\gamma<2$ by putting
\begin{equation}\label{uab}
a^\gamma+b^\gamma=1.
\end{equation}
For arbitrary distributions $A$ satisfying \eqref{prav}, condition \eqref{uab}
ensures asymptotic equivalence of the tails: ${\bar F}(x)\sim {\bar A}(x)$,
$x\to\infty$. We assume this condition to be fulfilled.

In this case it is in fact shown that $\theta=(1-a^\gamma)/(1-a^\gamma/\mu)\in (0,1)$
in the sense of Definition~\ref{def2} (by Property~\ref{pr5}), where $\mu>1$ is the
mean number of descendants. Note that $\theta$ tends to 1 both as the dependence
parameter $a$ decreases and as the mean number $\mu$ of descendants decreases. An
extremal index in the sense of Definition~\ref{def1} in this case does not exist.

Here we may also expect cluster forming. Clearly, this concerns groups of kindred
particles having a common ancestor with an abnormally large score and inheriting this
mutation. This conclusion is illustrated by computer simulation \cite{Leb7}.

\section{Models with Copulas}\label{sec5}

Recall some notions of copula theory [\citen{Nel}; \citen{QRM}, ch.~5 and
Section~7.5].

A copula $C$ is a multivariate distribution function on $[0,1]^d$, $d\ge 2$, if all
marginal distributions are uniform on $[0,1]$. By Sklar's theorem, any multivariate
distribution function in $\mathbb{R}^d$ can be represented as
$$
G(x_1,\ldots, x_d)=C(G_1(x_1),\ldots, G_d(x_d)),
$$
where $G_i$, $1\le i\le d$, are marginal distribution functions. Thus, to any
multivariate distribution there corresponds its copula. If the marginal distributions
are continuous, such a representation is unique.

To a vector with independent components there corresponds the independence copula
$$
C(y_1,\ldots,y_d)=y_1\ldots y_d.
$$

At present, the mathematical apparatus of copulas is actively used in quite diverse
applications and, in particular, spreads into information science and technology.
Note the paper \cite{Chat} on recursive neural networks, where Student, Clayton, and
Gumbel copulas were used. Based on them, successful learning of a humanlike robot was
performed.

In the general case, copulas may describe dependence in the behavior of components of
compound systems caused by their interaction or the influence of common external
factors. In models of the preceding sections, dependence of aggregate activities in a
network or particle scores in a generation can also be described by some copulas,
which, however, are hard to write out explicitly, making other methods preferable.
Note that communication of network users may lead to dependence of their individual
information activities, which was not taken into account in \cite{Leb4}. In
engineering systems, deterioration or breakage of some parts may affect other parts,
and all parts are influenced by a common operation regime (temperature, humidity,
etc.).

In financial models, copulas are used to describe the dependence between fluctuations
of exchange rates of various shares and currencies \cite{QRM}. This dependence must
be taken into account both in financial arrangements and in programming trading
(financial) bots (black boxes).

Below we will assume for simplicity that $\nu_n=n$ (triangular scheme), $F_n(x)\equiv
x$, $x\in [0,1]$, and random variables $\xi_{n,m}$, $1\le m\le n$, are related by an
$n$-variate copula $C_n$. Recall that we can pass to the uniform distribution from
any continuous distribution by Property~\ref{pr2}.

Let for any $s\in (0,1)$ the sequence $u_n(s)$ be such that $u_n(s)^n\to s$,
$n\to\infty$; then $u_n(s)=1+(1+o(1))(\ln s)/n$, $n\to\infty$.

\numberwithin{example}{section} \numberwithin{theorem}{section}
\numberwithin{corollary}{section}

\begin{example}\label{ex5.1} \textbf{Gumbel--Hougaard copula.}
This copula is of the form
$$
C(y_1,\ldots,y_d)=\exp\left\{-\left(\sum_{i=1}^d(-\ln
y_i)^\alpha\right)^{1/\alpha}\right\},\quad \alpha\ge 1,
$$
which implies
$$
C(y,\ldots,y)=y^{d^{1/\alpha}}.
$$
Assuming $C_n$ to be the Gumbel--Hougaard copula with $\alpha_n\ge 1$ and
$(\alpha_n-1)\ln n\to\gamma\ge 0$, we obtain
$$
{\bf P}(M_n\le u_n(s))=u_n(s)^{n^{1/\alpha_n}}\to s^\theta,\quad
\theta=e^{-\gamma}\in [0,1].
$$

This copula belongs to the class of extreme value (or max-stable) copulas. In the
general case, they have the Pickands representation \cite[p.~312, Theorem~7.45]{QRM}:
$$
C(y_1,\ldots,y_d)=\exp\left\{B\left(\frac{\ln y_1}{\sum_{i=1}^d\ln
y_i},\ldots,\frac{\ln y_1}{\sum_{i=1}^d\ln y_i}\right) \sum_{i=1}^d\ln y_i\right\},
$$
where
$$
B(w_1,\ldots, w_d)=\int_{S^d}\left(\bigvee_{i=1}^d x_iw_i\right)\,dH(x)
$$
and $H$ is a finite measure on $S^d=\{x=(x_1,\ldots, x_d): x_i\ge 0, \sum_{i=1}^d
x_d=1\}$. Moreover, this measure should be normalized so that
$\int_{S^d}x_i\,dH(x)=1$ for all $1\le i\le d$ (which was forgotten to be mentioned
in \cite{QRM}).

Note that the function $B$ is first-order homogeneous. Thus, in the general case we
have
$$
C(y,\ldots,y)=y^{B(1,\ldots, 1)}.
$$
Denote $\beta_n=B_n(1,\ldots, 1)$; thus, if $\beta_n/n\to\theta$, then $\theta$ is
the extremal index (in the sense of both definitions). Since $0\le\beta_n\le n$, we
have $\theta\in [0,1]$.
\end{example}

\begin{example}\label{ex5.2} \textbf{Clayton copula.}
This copula is of the form
$$
C(y_1,\ldots,y_d)=\left(\sum_{i=1}^d y_i^{-\alpha}-d+1\right)^{-1/\alpha},\quad
\alpha\ge 0,
$$
where the degenerate case $\alpha=0$ corresponds to the independence copula, arising
in the limit as $\alpha\to 0$. Hence,
$$
C(y,\ldots,y)=\left(d (y^{-\alpha}-1)+1\right)^{-1/\alpha}.
$$
Let $C_n$ be the Clayton copula with $\alpha_n\equiv \alpha>0$; then
$$
\begin{aligned}
{\bf P}(M_n\le u_n(s))&=\left(n(u_n(s)^{-\alpha}-1)+1\right)^{-1/\alpha}\\ & \to
\left(1-\alpha\ln s\right)^{-1/\alpha}=\psi(s).
\end{aligned}
$$
Here $\theta^-=0$ and $\theta^+=1$.
\end{example}

\begin{example}\label{ex5.3} \textbf{Frank copula.}
This copula is of the form
$$
C(y_1,\ldots,y_d)=-\frac{1}{\alpha}\ln\left(1-\frac{\prod_{i=1}^d(1-e^{-\alpha
y_i})}{(1-e^{-\alpha})^{d-1}}\right), \quad \alpha\ge 0,
$$
where the degenerate case $\alpha=0$ corresponds to the independence copula, arising
in the limit as $\alpha\to 0$. Hence,
$$
C(y,\ldots,y)=-\frac{1}{\alpha}\ln\left(1-\frac{(1-e^{-\alpha
y})^d}{(1-e^{-\alpha})^{d-1}}\right).
$$
Let $C_n$ be the Frank copula with $\alpha_n\equiv \alpha>0$; then, passing to the
limit, we obtain
$$
{\bf P}(M_n\le u_n(s))\to
-\frac{1}{\alpha}\ln\left(1-(1-e^{-\alpha})s^{\alpha/(e^\alpha-1)}\right)=\psi(s).
$$
In this case
$$
\lim_{s\to 0}\log_s\psi(s)=\alpha/(e^\alpha-1),\qquad \lim_{s\to 1}\log_s\psi(s)=1,
$$
and in the interval $(0,1)$ the function attains intermediate values. Therefore,
$\theta^-=\alpha/(e^\alpha-1)\in (0,1)$ and $\theta^+=1$.
\end{example}

All three examples deal with strictly Archimedean copulas. Recall that a copula is
said to be strictly Archimedean if it is of the form
\begin{equation}\label{copu}
C(y_1,\ldots,y_d)=\varphi^{-1}\left(\sum_{i=1}^d \varphi(y_i)\right),
\end{equation}
where $\varphi$ is a decreasing function on $[0,1]$, called the generator,
$\varphi(0)=+\infty$, $\varphi(1)=0$. For $d=2$, it suffices that the function is
convex. If we require the function $\varphi^{-1}$ to be completely monotone on
$(0,+\infty)$, then equation \eqref{copu} defines a copula for any $d\ge 2$
\cite[Theorem~4.6.2]{Nel}. Below we assume this condition on $\varphi$ to be
fulfilled.

On the other hand, the function $f$ is the Laplace--Stieltjes transform of some
distribution if and only if $f$ is completely monotone and $f(0)=1$ \cite[ch.~13,
Section~4, Theorem~1]{Fel}. Hence it follows that $\varphi^{-1}$ must be the
Laplace--Stieltjes transform of some distribution, and by the condition
$\varphi(0)=+\infty$ (and therefore $\varphi^{-1}(+\infty)=0$), this distribution
must have no atoms at zero. Thus, there exists a random variable $\zeta>0$ a.s.\ such
that
$$
\varphi^{-1}(u)={\bf E}e^{-u\zeta},\quad u\ge 0.
$$
Introduce the notation
$$
x_0=\inf\{x>0: {\bf P}(\zeta\le x)>0\},\qquad \mu={\bf E}\zeta.
$$
For brevity, we denote $f(u)=\varphi^{-1}(u)$.

\begin{theorem}\label{th5.1}
Let\/ $\mu<\infty$\textup; then we have the extremal function\/ $\psi(s)=f(-(\ln
s)/\mu)={\bf E}s^{\zeta/\mu}$, $\theta^+=1$, $\theta^-=x_0/\mu$.
\end{theorem}

\begin{proof}
Since $1-f(u)\sim \mu u$, $u\to 0+0$, we have $\varphi(1-t)\sim t/\mu$, $t\to 1-0$.
Next,
$$
\begin{aligned}
{\bf P}(M_n\le u_n(s)&=f(n\varphi(u_n(s)))\\ &=f(n\varphi(1+(1+o(1))(\ln s)/n)))\\
&\to f(-(\ln s)/\mu),\quad n\to\infty.
\end{aligned}
$$
From Jensen's inequality, we obtain $\psi(s)={\bf E}s^{\zeta/\mu}\ge s^{{\bf
E}\zeta}=s$. On the other hand, since $\zeta>0$ a.s., we have $\psi(s)\le
s^{x_0/\mu}$. Hence, $\theta^+\le 1$ and $\theta^-\ge x_0/\mu$. Furthermore, we
obtain
$$
\lim_{s\to 0}\log_s\psi(s)=x_0/\mu,\qquad\lim_{s\to 1}\log_s\psi(s)=1,
$$
so these estimates are attained at the limit and we have $\theta^+=1$ and
$\theta^-=x_0/\mu$.
\end{proof}

In the case of the Clayton copula, the generator is $\varphi(t)=t^{-\alpha}-1$, and
the inverse function $f(u)=1/(1+u)^{1/\alpha}$ corresponds to the gamma distribution
with shape parameter $1/\alpha$, for which $x_0=0$, so $\theta^-=0$ and $\theta^+=1$.

In the case of the Frank copula, the generator is $\varphi(t)=-\ln((1-e^{-\alpha
t})/(1-e^{-\alpha}))$, and the inverse function
$f(u)=-(1/\alpha)\ln(1-(1-e^{-\alpha})e^{-u})$ corresponds to the discrete
distribution with probabilities ${\bf P}(\zeta=k)=(1-e^{-\alpha})^k/(\alpha k)$,
$k\ge 1$. Then $x_0=1$ and $\mu=f'(0)=(e^\alpha-1)/\alpha$, whence
$\theta^-=\alpha/(e^\alpha-1)$ and $\theta^+=1$.

Among the considered examples, only that of the Gumbel--Hougaard copula does not
match the conditions of Theorem~\ref{th5.1}, since it has generator $\varphi(t)=(-\ln
t)^\alpha$ with the inverse function $f(u)=\exp\{-u^{1/\alpha}\}$, $\alpha\ge 1$,
which corresponds to an asymmetric $(1/\alpha)$-stable distribution on~$\mathbb{R}_+$
without a finite mean.

To analyze such cases, we apply the following modification.

Note that if $\varphi(t)$ is a generator with a completely monotone inverse function,
then $\varphi(t)^\beta$, $\beta\ge 1$, is also a generator with a completely monotone
inverse function \cite[Lemma~4.6.4]{Nel}.

\begin{theorem}\label{th5.2}
Assume that an $n$-variate copula $C_n$ has generator\/
$\varphi_n(t)=\varphi(t)^{\beta_n}$ with\/ $\beta_n\ge\nolinebreak 1$,
$(\beta_n-1)\ln n\to\gamma\ge 0$, and for the generator\/ $\varphi(t)$ we have\/
$\mu<\infty$. Then\/ $\psi(s)=f(-e^{-\gamma}(\ln s)/\mu)$,
$\theta^-=(x_0/\mu)e^{-\gamma}$, and\/ $\theta^+=e^{-\gamma}$.\sloppy
\end{theorem}

\begin{proof}
From $\varphi_n(t)=\varphi(t)^{\beta_n}$ it follows that $f_n(u)=f(u^{1/\beta_n})$.
We have
$$
\begin{aligned}
{\bf P}(M_n\le u_n(s))&=f_n(n\varphi_n(u_n(s)))\\
&=f(n^{1/\beta_n}\varphi(1+(1+o(1))(\ln s)/n)))\\ &=f(e^{((1-\beta_n)\ln
n)/\beta_n}(-(\ln s)/\mu))\\ &\to f(-e^{-\gamma}(\ln s)/\mu),\quad n\to\infty.
\end{aligned}
$$
Partial indices are obtained from the relation
$$
\log_s\psi(s)=\frac{\ln f(-e^{-\gamma}(\ln s)/\mu)}{\ln s}=e^{-\gamma}\frac{\ln
f(-(\ln r)/\mu)}{\ln r}, \quad r=s^{e^{-\gamma}}\in (0,1),
$$
where the fraction on the right-hand side is the logarithm of the extremal function
from Theorem~\ref{th5.1}, taking values from $x_0/\mu$ to $1$.
\end{proof}

In particular, the result of Example~\ref{ex5.1} for the Gumbel--Hougaard copula is
obtained with $\varphi_n(t)=(-\ln t)^{\alpha_n}$, where $\varphi(t)=-\ln t$
corresponds to $\zeta=1$ a.s.\ and the independence copula.

Thus, it is seen that in the considered models with copulas there may occur any
$\theta\in [0,1]$ and any $0\le\theta^-<\theta^+\le 1$.

Now let us consider one instructive example with copulas and random series sizes.

\begin{example}\label{ex5.4}
Let the series sizes satisfy the condition $\nu_n/n\stackrel{d}{\to}\zeta$,
$n\to\infty$, where $\zeta$ has a stable distribution with the Laplace--Stieltjes
transform ${\bf E}e^{-u\zeta}=e^{-u^\beta}$, $0<\beta<1$, and assume that in each
series the random variables (independent of $\nu_n$) are related by the
Gumbel--Hougaard copula with $\alpha_n>1$, $(\alpha_n-1)\ln n\to\gamma>0$,
$n\to\infty$ (see Example~\ref{ex5.1}).

First, assume that $u_n(s)^n\to e^{-\tau}$, $n\to\infty$, $\tau>0$; then
$$
{\bf E}u_n(s)^{\nu_n}={\bf E}(u_n(s)^n)^{\nu_n/n}\to {\bf
E}e^{-\tau\zeta}=e^{-\tau^\beta},\quad n\to\infty.
$$
Take $\tau=(-\ln s)^{1/\beta}$; then ${\bf E}u_n(s)^{\nu_n}\to s$, as required.

Next, we have
\begin{equation}\label{p335}
\begin{aligned}[b]
{\bf P}(M_n\le u_n(s))&=u_n(s)^{\nu_n^{1/\alpha_n}}
=\left(u_n(s)^{n^{1/\alpha_n}}\right)^{(\nu_n/n)^{1/\alpha_n}}\\ &\to{\bf E}e^{-\tau
e^{-\gamma}\zeta}=e^{-(e^{-\gamma}\tau)^\beta}=s^{e^{-\gamma\beta}},\quad n\to\infty.
\end{aligned}
\end{equation}
Hence, the extremal index in the sense of Definition~\ref{def1} is
$e^{-\gamma\beta}$.

On the other hand, for any $\theta>0$ we have
$$
{\bf E}u_n(s)^{\theta\nu_n}\to {\bf
E}e^{-\tau\theta\zeta}=e^{-(\theta\tau)^\beta}=s^{\theta^\beta},\quad n\to\infty,
$$
from which together with \eqref{p335} it follows that the extremal index in the sense
of Definition~\ref{def2} is~$e^{-\gamma}$.

Thus, the system has two \emph{different\/} extremal indices according to two
different definitions!
\end{example}

In all the previously considered models, the classical property \eqref{Mhat} remained
valid in the form $\psi(s)\ge s$ for all $s\in [0,1]$. To conclude with, consider an
example where it is violated. In this example, the symmetric dependence of random
variables in a series can be described by a~copula, but it is simpler to do this
constructively, using a direct construction.

\begin{example}\label{ex5.5}
Let $\eta_{n,m}$, $m\ge 1$, $n\ge 1$, be independent and have the uniform
distribution on~$[0,1]$; $\nu_n=n$; $\kappa_n$ take values from $1$ to $n$
equiprobably, being independent of $\eta_{n,m}$, $1\le m\le n$; $\gamma>0$. Put
$$
\xi_{n,m}=
\begin{cases}
\eta_{n,m}^{1/(\gamma n)}, & m=\kappa_n,\\ \eta_{n,m}, & m\ne\kappa_n.
\end{cases}
$$
Then the joint distribution function of $\xi_{n,m}$, $1\le m\le n$, is of the form
$$
F^{(n)}(x_1,\ldots, x_n)=\left(\prod_{m=1}^n x_m\right)\left(\frac1n\sum_{m=1}^n
x_m^{\gamma n-1}\right),
$$
whence
$$
F_n(x)=x\left(1+\frac{x^{\gamma n-1}-1}{n}\right),\qquad {\bf P}(M_n\le
x)=x^{(1+\gamma)n-1}.
$$
Letting $u_n(s)=1-(1+o(1))\tau/n$, $\tau>0$, we obtain as $n\to\infty$
$$
F_n(u_n(s))^n\to e^{-\tau}\exp\{e^{-\gamma\tau}-1\}=s,\qquad {\bf P}(M_n\le
u_n(s))\to e^{-(1+\gamma)\tau}=\psi(s),
$$
whence we can explicitly find the inverse extremal function
$$
\psi^{-1}(u)=u^{1/(1+\gamma)}\exp\{u^{\gamma/(1+\gamma)}-1\},
$$
for which we have $\psi^{-1}(u)>u$ for all $u\in (0,1)$ and hence $\psi(s)<s$ for all
$s\in (0,1)$. In this case it can be shown that $\theta^-=1$ and
$\theta^+=1+\gamma>1$.
\end{example}

\section{Threshold Models}\label{sec6}

Up to now, we considered models where $\nu_n$ was determined by reasons external with
respect to the variables $\{\xi_{n,m}\}$. Now we introduce models where $\nu_n$ is
the stopping time with respect to the sequence $\{\xi_{n,m}$, $m\ge 1\}$, where
$\xi_{n,m}$, $m\ge 1$, are independent and uniformly distributed on~$[0,1]$, and
stopping occurs when a current random variable passes over some threshold value.

Such models may occur, for example, in problems of automated search for objects
possessing certain properties by simple exhaustive search.

Note that papers \cite{Novak-2013, Novak-1991} considered a model of maxima of random
variables where stopping occurred at the time of threshold passage (namely, time $t$)
by not the current variable but their accumulated sum. This is a generalization of
the classical problem on the longest success series in Bernoulli trials
\cite[Section~8.5]{EKM}. However, in this case the stopping time simply grows
asymptotically proportionally to the threshold, and there are no such interesting
effects as in the models in question.

\begin{example}\label{ex6.1}
Consider a number sequence $a_n\in (0,1)$, $n\ge 1$, with $a_n\to 1$, $n\to \infty$.
Denote $\varepsilon_n=1-a_n>0$; then $\varepsilon_n\to 0$, $n\to\infty$. Put
$\nu_n=\min\{m\ge 1: \xi_{n,m}>a_n\}$. Then
\begin{equation}\label{mot}
\begin{aligned}[b]
{\bf P}(M_n\le u_n(s))&={\bf P}(\xi_{n,\nu_n}\le u_n(s))\\ &={\bf P}(\xi_{1,1}\le
u_n(s)|\xi_{1,1}>a_n)\\ &=0\vee (u_n(s)-a_n)/\varepsilon_n\\ &=0\vee
(1-(1-u_n(s))/\varepsilon_n),
\end{aligned}
\end{equation}
where $u_n(s)$ are determined by the condition
\begin{equation}\label{eot}
{\bf E}u_n(s)^{\nu_n}=\frac{\varepsilon_n u_n(s)}{1-(1-\varepsilon_n)u_n(s)}\to
s,\quad n\to\infty,
\end{equation}
since $\nu_n$ have the geometric distribution (starting from 1) with parameter
$\varepsilon_n$. Equation \eqref{eot} implies
\begin{equation}\label{1mun}
1-u_n(s)\sim \varepsilon_n\frac{1-s}{s},\quad n\to\infty.
\end{equation}
Substituting this into \eqref{mot} and passing to the limit, we obtain $\psi(s)=0\vee
(2-1/s)$ according to Definition~\ref{def1}. In this case, as in Example~\ref{ex5.5},
we have $\psi(s)<s$ for all $s\in (0,1)$. We have $\psi(s)=0$, and therefore
$\log_s\psi(s)=+\infty$ for $s\in [0,1/2]$, or $\log_s\psi(s)>1$ for $s\in (1/2,1)$,
and also $\log_s\psi(s)\to 1$, $s\to 1$. Hence, $\theta^-=1$ and $\theta^+=+\infty$.

Surprisingly, the result does not depend on the choice of a sequence $a_n$, $n\ge 1$.

What can be said in this case about the extremal index in the sense of
Definition~\ref{def2}? From~\eqref{eot}, taking into account \eqref{1mun}, we obtain
$$
{\bf E}u_n(s)^{\theta\nu_n}=\frac{\varepsilon_n
u_n(s)^\theta}{1-(1-\varepsilon_n)u_n(s)^\theta}\to
\frac{s}{\theta+(1-\theta)s},\quad n\to\infty,
$$
but the extremal function is not of this form; hence, there is no extremal index in
the sense of Definition~\ref{def2}.
\end{example}

Now consider a model with random thresholds $\zeta_n$, $n\ge 1$. Let $0<\zeta_n<1$
a.s.; $\xi_{n,m}$, $m\ge 1$, are independent of $\zeta_n$, and $\nu_n=\min\{m\ge 1:
\xi_{n,m}>\zeta_n\}$.

\begin{theorem}\label{th6.1}
Let\/ $n(1-\zeta_n)\stackrel{L_1}{\to}\zeta>0$, $n\to\infty$, ${\bf E}\zeta=1$.
Then\/ $\psi(s)=g(f^{-1}(s))$, where $f(t)={\bf E}(\zeta/(t+\zeta))$ and\/ $g(t)={\bf
E}(\zeta-t)_+$.
\end{theorem}

\begin{proof}
Given that $\zeta_n=x\in (0,1)$, the series size $\nu_n$ has geometric distribution
with parameter $1-x$. Hence it follows that $(1-\zeta_n)\nu_n\stackrel{d}{\to}\eta$
and $\nu_n/n\stackrel{d}{\to}\eta/\zeta$, $n\to\infty$, where $\eta$ has standard
exponential distribution and is independent of $\zeta$. Denote $f(t)={\bf
E}e^{-t\eta/\zeta}$; then $f(t)={\bf E}(\zeta/(t+\nolinebreak\zeta))$.

Let $\tau>0$; then
$$
{\bf E}(1-\tau/n)^{\nu_n}={\bf E}((1-\tau/n)^n)^{\nu_n/n}\to {\bf
E}e^{-\tau\eta/\zeta}=f(\tau),\quad n\to\infty.
$$
Thus, ${\bf E}u_n(s)^{\nu_n}\to s$ implies $u_n(s)=1-(1+o(1))f^{-1}(s)/n$,
$n\to\infty$.

We obtain
$$
\begin{aligned}
{\bf P}(M_n\le u_n(s))&={\bf P}(\xi_{1,1}\le u_n(s)|\xi_{1,1}>\zeta_n)\\ &={\bf
P}(\zeta_n<\xi_{1,1}\le u_n(s))/\varepsilon_n\\ &={\bf
E}(1-(1+o(1))f^{-1}(s)/n-\zeta_n)_+/\varepsilon_n\\ &\to {\bf
E}(\zeta-f^{-1}(s))_+,\quad n\to\infty,
\end{aligned}
$$
as required.
\end{proof}

Recall a formula
\begin{equation}\label{forepl}
{\bf E}(\zeta-t)_+=\int_t^{+\infty}{\bar F}_\zeta(x)\,dx,
\end{equation}
convenient for computations, which is obtained via integration by parts.

\begin{example}\label{ex6.2}
Let $\zeta$ equiprobably take values $1-\delta$ and $1+\delta$, $0<\delta<1$ (the
case $\delta=0$ reduces to Example~\ref{ex6.1}). Then
$$
f(t)=\frac12\left(\frac{1}{1+t/(1-\delta)}+\frac{1}{1+t/(1+\delta)}\right)
=\frac{(1+t)-\delta^2}{(1+t)^2-\delta^2},
$$
whence
$$
f^{-1}(s)=\frac{1+\sqrt{1-4s(1-s)\delta^2}}{2s}-1
$$
and
$$
\psi(s)=\frac12(1-\delta-f^{-1}(s))_++\frac12(1+\delta-f^{-1}(s))_+.
$$

We have$\psi(s)=0$ for $0<s<f(1+\delta)=(2-\delta)/4$, so $\theta^+=+\infty$.
\end{example}

In the general case one can analyze the asymptotic behavior of the extremal function
as $s\to 0$ and $s\to 1$. Denote
$$
\theta_0=\lim_{s\to 0}\log_s\psi(s),\qquad \theta_1=\lim_{s\to 1}\log_s\psi(s).
$$
One may claim that $\theta^-\le \theta_0\wedge \theta_1$ and $\theta^+\ge
\theta_0\vee \theta_1$.

\begin{corollary}\label{cor6.1}
\strut
\begin{enumerate}[\rm(1)]\addtolength{\itemsep}{-5pt}
\item\vskip-5pt
If\/ ${\bar F}_\zeta(x)\sim Cx^{-\alpha}$, $x\to\infty$, $C>0$, $\alpha>1$, then\/
$\psi(s)\sim Cs^{\alpha-1}/(\alpha-1)$, $s\to 0$, and\/ $\theta_0=\alpha-1$. If
${\bar F}_\zeta(x)$ decreases faster than any power, then\/ $\theta_0=+\infty$.
\item
If\/ ${\bf E}\zeta^{-1}<\infty$, then\/ $\theta_1=1/{\bf E}\zeta^{-1}$. If\/ ${\bf
E}\zeta^{-1}=\infty$, then\/ $\theta_1=0$.
\end{enumerate}
\end{corollary}

\begin{proof}
(1) Note that $f(t)={\bf E}(\zeta/(\zeta+t))\sim {\bf E}\zeta/t=1/t$, $t\to\infty$.
Therefore, $f^{-1}(s)\sim 1/s$, $s\to 0$. Now ${\bar F}_\zeta(x)\sim Cx^{-\alpha}$,
$x\to\infty$, implies by equation \eqref{forepl} that $g(t)\sim
Ct^{-(\alpha-1)}/(\alpha-1)$, $t\to\infty$. Hence, $\psi(s)=g(f^{-1}(s))\sim
Cs^{\alpha-1}/(\alpha-1)$, $s\to 0$, and $\theta_0=\alpha-1$.

If ${\bar F}_\zeta(x)=o(x^N)$, $x\to\infty$, then $g(t)=o(t^{N-1})$, $t\to\infty$,
and $\psi(s)=o(s^{N-1})$, $s\to 0$, for any $N>0$, whence $\theta_0=+\infty$.

(2) For ${\bf E}\zeta^{-1}<\infty$ we have $1-f(t)={\bf E}(t/(\zeta+t))\sim t{\bf
E}\zeta^{-1}$, $t\to 0$. Hence, $f^{-1}(s)\sim (1-s)/{\bf E}\zeta^{-1}$, $s\to 1$.
Furthermore, $1-g(t)\sim t$, $t\to 0$. Therefore, $1-\psi(s)\sim (1-s)/{\bf
E}\zeta^{-1}$, $s\to 0$, whence $\theta_1=1/{\bf E}\zeta^{-1}$. The result for ${\bf
E}\zeta^{-1}=\infty$ is obtained by passing to the limit.
\end{proof}

By Corollary~\ref{cor6.1} in Example~\ref{ex6.2} we obtain $\theta_0=+\infty$ and
$\theta_1=1-\delta^2\in (0,1)$.

It is clear that if one of the exponents $\theta_0$ and $\theta_1$ is greater than 1
and the other is smaller, then the graph of $\psi(s)$ necessarily crosses the
diagonal. Then property \eqref{Mhat} in the form $\psi(s)\ge s$ holds for some $s\in
(0,1)$ and is violated for some other.

\section{Conclusion}\label{sec7}

We have generalized the notion of the extremal index of a stationary random sequence
to a~series scheme with random lengths (by two definitions). We have studied
properties of the new extremal indices. We have considered various applications of
these indices to models of information networks and biological populations, models
with copulas, and threshold models. We gave examples where there exist both extremal
indices, only one of them, or none. In cases where the extremal index in the sense of
the first definition does not exist, we have found partial indices. Thus, we have
made a number of important steps in constructing a new mathematical apparatus, which
is of both theoretical and applied importance in the description of extremal behavior
of various systems.

Of course, the research of the new extremal indices, their properties, and their
applications cannot be covered by a single paper. This paper is rather intended to
open a cycle of papers---or perhaps a whole scientific direction---that other
researchers can join, as it happens in the study of the classical extremal index.

\section*{Acknowledgements}

The research was supported by the Russian Foundation for Basic Research, project
no.~14-01-00075.

\medskip
{\large Contributors}

~

{\bf Lebedev Alexey V.} (b. 1971) --- Dr. Sci. in physics and mathematics, associate
professor, Department of Probability Theory, Faculty of Mechanics and Mathematics,
Lomonosov Moscow State University, Leninskiye Gory, Moscow 119991, Russian
Federation; \url{avlebed@yandex.ru}

\end{document}